\documentclass[11pt,english]{amsart}  

\usepackage{amssymb,url,xspace,amsthm}
\usepackage{mathrsfs}

\usepackage{paralist}
\usepackage{datetime} 
\usepackage{colortbl}

\usepackage[T1]{fontenc}
\usepackage{newtxtext}
\usepackage{newtxmath}
\usepackage{textcomp}
\usepackage[dvipsnames]{xcolor}
\usepackage{mathtools}
\usepackage{euscript}
\usepackage[all]{xy}
\usepackage{hyperref}
\usepackage{graphicx} 
\usepackage[text={155mm,250mm},centering]{geometry} 
\usepackage{comment}  
\usepackage{cite}  
\usepackage{thmtools}
\usepackage{enumitem}
\usepackage{letltxmacro} 
\usepackage{nameref}
\usepackage{cleveref}
\usepackage{calc}
\usepackage{interval}
\usepackage{manfnt}
\usepackage{tikz-cd}  
\usepackage[utf8]{inputenc}
\usepackage[hyperpageref]{backref}

\usepackage{faktor} 
\usepackage{bm}

 \setlist[itemize]{wide = 0pt, labelwidth = 2em, labelsep*=0em, itemindent = 0pt, leftmargin = \dimexpr\labelwidth + \labelsep\relax, noitemsep,topsep = 1ex,}
 \setlist[enumerate]{wide = 0pt, labelwidth = 2em, labelsep*=0em, itemindent = 0pt, leftmargin = \dimexpr\labelwidth + \labelsep\relax, noitemsep,topsep = 1ex}

\theoremstyle{plain}
\newtheorem{thmx}{Theorem}
\renewcommand{\thethmx}{\Alph{thmx}} 
\newtheorem{thm}{Theorem}[section]  
\newtheorem{lem}[thm]{Lemma}

\newtheorem{proposition}[thm]{Proposition}

\theoremstyle{definition}
\newtheorem{dfn}[thm]{Definition}

\theoremstyle{remark}

\numberwithin{equation}{subsection}  

\theoremstyle{plain}
\newlist{thmlist}{enumerate}{1}
\setlist[thmlist]{wide = 0pt, labelwidth = 2em, labelsep*=0em, itemindent = 0pt, leftmargin = \dimexpr\labelwidth + \labelsep\relax, noitemsep,topsep = 1ex, font=\normalfont, label=(\roman*), ref=\thethm.(\roman{thmlisti})}

\addtotheorempostheadhook[thm]{\crefalias{thmlisti}{thm}}

\addtotheorempostheadhook[assumpsion]{\crefalias{thmlisti}{assumption}}

\addtotheorempostheadhook[cor]{\crefalias{thmlisti}{cor}}

\addtotheorempostheadhook[proposition]{\crefalias{thmlisti}{proposition}}

\addtotheorempostheadhook[dfn]{\crefalias{thmlisti}{dfn}}

\addtotheorempostheadhook[lem]{\crefalias{thmlisti}{lem}}
\addtotheorempostheadhook[main]{\crefalias{thmlisti}{main}}

\addtotheorempostheadhook[rem]{\crefalias{thmlisti}{rem}}

\newlist{thmenum}{enumerate}{1} 
\setlist[thmenum]{wide = 0pt, labelwidth = 2em, labelsep*=0em, itemindent = 0pt, leftmargin = \dimexpr\labelwidth + \labelsep\relax, noitemsep,topsep = 1ex, font=\normalfont, label=(\roman*), ref=\thethmx.(\roman{thmenumi})}
\crefalias{thmenumi}{thmx} 

\newlist{corlist}{enumerate}{1} 
\setlist[corlist]{wide = 0pt, labelwidth = 2em, labelsep*=0em, itemindent = 0pt, leftmargin = \dimexpr\labelwidth + \labelsep\relax, noitemsep,topsep = 1ex, font=\normalfont, label=(\roman*), ref=\thecorx.(\roman{corlisti})}
\crefalias{corlisti}{corx} 
 
\crefname{lem}{Lemma}{Lemmas} 
\crefname{conjecture}{Conjecture}{Conjectures}
\crefname{thm}{Theorem}{Theorems}
\crefname{proposition}{Proposition}{Propositions}
\crefname{dfn}{Definition}{Definitions}
\crefname{rem}{Remark}{Remarks}
\crefname{cor}{Corollary}{Corollaries}
\crefname{corx}{Corollary}{Corollaries}
\crefname{problem}{Problem}{Problems}
\crefname{thmx}{Theorem}{Theorems}
\crefname{claim}{Claim}{Claims}
\crefname{assumption}{Assumption}{Assumptions}
\crefname{main}{Main Theorem}{Main Theorems}

\def\ep{\varepsilon}

\def\Ad{{\rm Ad}\,}

\makeatletter
\newcommand*{\rom}[1]{\expandafter\@slowromancap\romannumeral #1@}
\makeatother

\makeatletter     
\newcommand{\crefnames}[3]{%
	\@for\next:=#1\do{%
		\expandafter\crefname\expandafter{\next}{#2}{#3}%
	}%
}
\makeatother

\crefnames{part,chapter,section}{\S}{\S\S}

\newcommand{\sC}{\mathscr{C}}
\newcommand{\sD}{\mathscr{D}}

\newcommand{\cO}{\mathcal O}
\newcommand{\cP}{\mathcal P}
\newcommand{\cV}{\mathcal V}

\newcommand{\bC}{\mathbb{C}}
\newcommand{\bD}{\mathbb{D}}
\newcommand{\bE}{\mathbb{E}}

\newcommand{\bH}{\mathbb{H}}

\newcommand{\bR}{\mathbb{R}}

\newcommand{\bZ}{\mathbb{Z}}

\newcommand{\kg}{\mathfrak{g}}

\def\db{\bar{\partial}}

 \def\d{\partial} 
\def\hess{i\partial\bar{\partial}}

\def\End{{\rm \small  End}}

\newcommand{\pae}{\mathcal{P}_{\!\bm{\alpha}}E}

 \newcommand{\pa}{\mathcal{P}_{\!\bm{\alpha}}}
 
  \newcommand{\ide}{{\rm Id}}

\begin{document} 
\title[Nilpotent orbit theorem]{On the nilpotent orbit theorem of \\ complex variations of Hodge structure} 


\author{Ya Deng}  

\address{CNRS, Institut \'Elie Cartan de Lorraine, Universit\'e de Lorraine, 54000 Nancy,
	France.}

\email{ya.deng@math.cnrs.fr} 

\urladdr{https://ydeng.perso.math.cnrs.fr} 
\begin{abstract}  
	We prove  some results on   the nilpotent orbit theorem for complex variations of Hodge structure. 
\end{abstract}  

	\maketitle
\tableofcontents	
	\section{Introduction} 
The nilpotent and ${\rm SL}_2$-orbit theorems of Schmid have been fundamental in understanding the degeneration of Hodge structure, particularly in the context of integral variation of Hodge structures. However, their complete generalization to complex variations of Hodge structure remains unproven. This paper aims to  the study of Schmid's nilpotent orbit theorem for complex variations of Hodge structure. The main result of this paper is the main component of the nilpotent orbit theorem.
	\begin{thmx}\label{main}
		Let $X$ be a complex manifold and let $D=\sum_{i=1}^\ell D_i$ be a simple normal crossing divisor on $X$.  Let $(V, \nabla, F^\bullet, Q)$ be a complex polarized variation of Hodge structure  on $X\backslash D$. Then   for any multi-index $\bm{\alpha}=(\alpha_1,\ldots,\alpha_\ell)\in \bR^\ell$,  $F^p_{\bm{\alpha}}:=j_*F^p\cap V_{\bm{\alpha}}^{\rm Del}$ and $F^{p}_{\bm{\alpha}}/F^{p+1}_{\bm{\alpha}}$ are both locally free sheaves. Here  $V_{\bm{\alpha}}^{\rm Del}$  is  the  Deligne extension of the flat bundle  $(V,\nabla)$ with the eigenvalues of the residue of $\nabla$ over $D_i$ lying in $[-\alpha_i,-\alpha_i+1)$. 
	\end{thmx} 
We prove moreover that the grading $\oplus_{p+q=m}F^{p}_{\bm{\alpha}}/F^{p+1}_{\bm{\alpha}}$ is naturally identified with $\oplus_{p+q=m}\pa E^{p,q}$ where $\pa E^{p,q}$ is the  prolongation of the Hodge bundles $ E^{p,q}:=F^{p}/F^{p+1}$ in terms of the norm growth of the Hodge metric (see \cref{sec:pro} for the definition).

Based on \cref{main}, we can generalize  main parts of Schmid's nilpotent orbit theorem to complex polarized variation of Hodge structure.
\begin{thmx}\label{corx}
	Let $(V,\nabla,F^\bullet,Q)$ be a complex polarized variation of Hodge structure on $(\Delta^*)^p\times \Delta^q$. Denote by $\Phi:\bH^{p}\times \Delta^q\to \sD$ its period mapping, where $\sD$ is the period domain and  $\bH=\{z\in \bC\mid \Re z<0\}$. Let us denote by $2\pi iR_i$ is the logarithm of the monodromy operator associated to the counter-clockwise generator of the
	fundamental group of the $i$-th copy of $\Delta^*$ 	in $(\Delta^*)^p$,  whose eigenvalues lie in $(2\pi i (\alpha_i-1), 2\pi i \alpha_i]$ for some $\bm{\alpha}\in \bR^{p}$.  Then for the holomorphic mapping $\Psi:(\Delta^*)^p\times \Delta^q\to \check{\sD}$ induced by $\tilde{\Psi}:= \exp(\sum_{i=1}^{p}z_iR_i)\circ \Phi(z,w)$, 
	\begin{thmenum}
		\item  \label{extension}$\Psi$ extends holomorphically to $\Delta^{p+q}$;
	 	\item \label{horizontal}the holomorphic mapping \begin{align*}
	 \vartheta:	\bH^{p}\times \Delta^q&\to \check{\sD}\\
	 		 (z,w)&\mapsto \exp(-\sum_{i=1}^{p}z_iR_i)\circ a(w) 
 		 \end{align*} is horizontal, where $a(w):=\Psi(0,w)$, and $\check{\sD}$ is the compact dual of the period domain $\sD$. 
	 	\item \label{one var}In the one variable case,  $  \exp(-zR)\circ a$  lies in $\sD$ when $\Re z\leq -C$ for some $C>0$. Moreover, we have the distance estimate
	 	$$
	 	d_{ {\sD}}( \exp(-zR)\circ a, \Phi(z))\leq C' |\Re z|^{\beta}e^{\delta \Re z} \quad\mbox{for some}\quad C',\delta,\beta>0
	 	$$
if  $\Re z\leq -C$.
		\end{thmenum}
\end{thmx}

When $(V,\nabla) $ has quasi-unipotent monodromies around $D$, \cref{main,corx} are contained in   Schmid's nilpotent orbit theorem \cite{Sch73}. Under this monodromy assumption he   proved \cref{one var} for the case  of several variables.  

\cref{main,corx} were also proved by Sabbah and Schnell  \cite{SS22} for the case  of one variable in a different way. Their  methods can be extended to prove the general case. 

Our proof of \cref{main} is based on Mochizuki's work on the prolongation of acceptable bundles \cite{Moc11} and methods in $L^2$-estimates.   The proof of \cref{horizontal,one var} essentially follows Schmid's method in \cite{Sch73}.

\medspace

We conclude the introduction by explaining  applications of the main result in this paper. One application is related to the work \cite{Wu22} on the  injectivity and vanishing theorem for $\bR$-Hodge
modules. We remark that in \cite[Lemma 2.10]{Wu22}, nilpotent orbit theorem for real variation of Hodge structures was claimed. It seems to the author that the proof needs some amplification, and as such, \cref{main} complements the proof presented therein.

Another  potential further application is on the complex Hodge modules.   It is well-established that Saito's theory of mixed Hodge modules fundamentally relies on the nilpotent orbit theorem. Currently,  Sabbah and Schnell are  developing the theory of mixed Hodge modules for complex Hodge structures. We remark that the nilpotent orbit theorem established in this paper should serve as a foundational component in their work.

	\section{Preliminary}
	\subsection{Complex polarized variations of Hodge structure}\label{sec:VHS}
Let us briefly recall the definition of \emph{polarized Hodge structure}. We refer the readers to \cite{Sim92,SS22} for more details. A Hodge structure of weight $w$ on a complex vector space $V$ is a decomposition
	$$
	V=\bigoplus_{p+q=w} V^{p, q}
	$$
	and a \emph{polarizatio}n is a hermitian pairing $Q: V \otimes_{\mathbb{C}} \bar{V} \rightarrow \mathbb{C}$ such that the above decomposition is orthogonal with respect to $Q$, and such that $(-1)^q Q$ is positive definite on the subspace 
	$V^{p, q}$.   For any $v\in V$, its \emph{Hodge norm}   is defined 
	$$
|v|^2=\sum_{p+q=w}(-1)^q Q\left(v^{p, q}, v^{p, q}\right),
	$$
	where $v^{p,q}$ is the $(p,q)$-component of $v$ in the Hodge decomposition $V=\bigoplus_{p+q=w} V^{p, q}$. 
Such a polarized Hodge structure induces the  \emph{Hodge filtration} $F^\bullet V$ of $V$ defined by
	$$
 F^p V=\bigoplus_{i \geq p} V^{i, w-i}.
	$$

As introduced by \cite{Sim92}, a complex polarized variation of Hodge structure    $(V=\bigoplus_{r+s=w} V^{r, s}, \nabla,  Q)$ of weight $w$ on a complex manifold $U$ consists of the following data:    
  \begin{enumerate}[label=(\alph*)]
	\item  a smooth vector bundle $V$ with a Hodge decomposition $V=\bigoplus_{r+s=w} V^{r, s}$;
	\item a flat connection $\nabla$   satisfies the Griffiths  transversality condition
\begin{align}\label{eq:GT}
\nabla:  {V}^{r, s} \rightarrow A^{0,1}\left(V^{r+1, s+1}\right) \oplus  A^{1,0}\left(V^{r, s}\right) \oplus A^{0,1}\left(V^{r, s}\right) \oplus {A}^{1,0}\left(V^{r-1, s+1}\right)
\end{align}
	\item\label{item:para}  a parallel Hermitian form $Q$ which makes the Hodge decomposition orthogonal and which on $\mathrm{V}^{r, s}$ is positive definite if $r$ is even and negative definite if $r$ is odd. 
\end{enumerate}
We decompose $\nabla=\bar{\theta}+\partial+\bar{\partial}+\theta$ according to the above transversality condition in \eqref{eq:GT}. The most important component of the connection turns out to
be the \emph{Higgs field}, which is the linear operator
$$
\theta: V^{r,s}\to A^{1,0}(V^{r-1,s+1}).
$$

We decompose  $\nabla=\nabla'+\nabla''$ into its (1,0)-component $\nabla': A^0(V) \rightarrow A^{1,0}(V)$ and its $(0,1)$-component $\nabla'': A^0(V) \rightarrow A^{0,1}(V)$. Then $\nabla''$ gives $V$ the structure of a holomorphic vector bundle, which we denote by the symbol $\cV$; and $\nabla'$ defines an integrable holomorphic connection $\nabla': \cV \rightarrow \Omega_U^1 \otimes_{\mathscr{O}_U} \cV$ on this bundle. The condition on $\nabla$ in \eqref{eq:GT} is saying that the Hodge filtration
$$
F^p V=\bigoplus_{i \geq p} V^{i, w-i}
$$
is a  holomorphic subbundles of $\cV$, and that the connection $\nabla'$ satisfies Griffiths' transversality relation $$\nabla'\left(F^p \cV\right) \subseteq \Omega_U^1 \otimes_{\mathscr{O}_U} F^{p-1} \cV.$$
Note that $(V^{p,q},\bar{\d})$ defines a holomorphic vector bundle, denoted by $E^{p,q}$.  
From the above point of view, the Higgs field is simply the holomorphic operator
$$
\theta: E^{p,q} \rightarrow \Omega_U^1 \otimes_{\mathscr{O}_U}  E^{p-1,q+1}. 
$$ 
  Denote by $(E,\theta)=(\oplus_{p+q=w}E^{p,q}, \theta)$, which is called the \emph{system of Hodge bundles} relative to the $\bC$-VHS $(V=\oplus_{p+q=w}V^{p,q},\nabla,Q)$.  Let us define $h_{p,q}:=(-1)^pQ$, which is a hermitian metric for $E^{p,q}$ by \Cref{item:para}.     For the hermitian metric $h=\oplus_{p+q=w}h_{p,q}$ of $E$, the component $\bar{\theta}:V^{r,s}\to A^{0,1}(V^{r+1,s-1})$  is the  adjoint of $\theta$  with respect to $h$.    Since the primary objective of this paper is to establish the nilpotent orbit theorem for $\bC$-VHS, we will adopt the notation $(V, \nabla, F^\bullet V, Q)$ to represent the $\bC$-VHS, as the Hodge filtration $F^\bullet V$ plays a central role in the nilpotent orbit theorem.

	\subsection{Deligne extension}\label{sec:Del}
		Let $X$ be a  complex manifold and let $D$ be a simple normal crossing divisor on $X$.  For the flat bundle   $(V, \nabla)$ defined on $U:=X\backslash D$, Deligne introduced a way to extend it across $D$. We recall this construction briefly and refer the readers to \cite{Sch73,Moc07,SS22} for more details.  For any point $x\in D$, we choose an admissible coordinate $(\Omega;z_1,\ldots,z_n)$ such that $\Omega\simeq \Delta^{n}$ and $D\cap \Omega=(z_1\ldots z_p=0)$ (see \cref{def:admissible} for the definition).  Write $q=n-p$. 
		The fundamental group $\pi_1\big((\Delta^*)^p\times \Delta^q\big)$ is generated by elements
	$\gamma_1,\ldots,\gamma_p$, where $\gamma_j$ may be identified with the counter-clockwise generator of the
	fundamental group of the $j$-th copy of $\Delta^*$
	in $(\Delta^*)^p$. We denote by $V^\nabla$ the space of multivalued flat sections of $(V,\nabla)$, which is a finite dimensional $\bC$-vector space. Set $T_j$ to be the 
	monodromy transformation with respect to $\gamma_j$, which  pairwise commute and are endomorphisms of $V^\nabla$; that is, for any multivalued section $v(t_1,\ldots,t_{p+q})\in V^\nabla$, one has
	$$
	v(t_1,\ldots,e^{2\pi i}t_j,\ldots,t_{p+q})=(T_j v)(t_1,\ldots,t_{p+q})
	$$
	and $[T_j,T_k]=0$ for any $j,k=1,\ldots,p$.  Let us write $Sp(T_j)$ the set of   eigenvalues of $T_j$, and for any $\lambda_j\in Sp(T_j)$, we denote by $\bE(T_j,\lambda_j)\subset V^\nabla$ the corresponding  eigenspace.  We know that all $\lambda_j\in Sp(T_j)$ has norm $1$ (see \emph{e.g.} \cite{SS22}).   Write $Sp:=\prod_{i=1}^pSp(T_j)$. For $\bm{\lambda}=(\lambda_1,\ldots,\lambda_p)$, we define
	$$
	\bE_{\bm{\lambda}}:=\cap_{j=1}^{p}\bE(T_j,\lambda_j).
	$$
		Since $T_j$   pairwise commute,  one has
		$$
		V^\nabla=\oplus_{\bm{\lambda}\in Sp} \bE_{\bm{\lambda}},
		$$
	and $\bE_{\bm{\lambda}}$ is an invariant subspace of $T_j$ for any $\bm{\lambda}\in Sp$ and any $j$.
	
	Let us fix a $p$-tuple $\bm{\alpha}:=(\alpha_1,\ldots,\alpha_p)\in \bR^p$. Then for $\bm{\lambda}\in Sp$, there exists a unique $\beta_i\in (\alpha_i-1, \alpha_i]$ such that $\exp(2\pi i \beta_i)=\lambda_i$.  Since $\lambda_i^{-1}T_i|_{ \bE_{\bm{\lambda}}}$ is unipotent, its logarithm can be defined as
	$$
\log(	\lambda_i^{-1}T_i|_{ \bE_{\bm{\lambda}}}):=\sum_{k=1}^{\infty}(-1)^{k+1}\frac{(\lambda_i^{-1}T_i|_{ \bE_{\bm{\lambda}}}-I)^k}{k}.
	$$
	We denote  $N_i:=\frac{\log(	\lambda_i^{-1}T_i|_{ \bE_{\bm{\lambda}}})}{2\pi i}$.  Then for any $v\in \bE_{\bm{\lambda}}$, we define 
\begin{align}\label{eq:twist}
\tilde{v}(t):=\exp\big(- \sum_{i=1}^{p}( \beta_iI+N_i)\cdot \log t_i  \big)v(t) = \prod_{i=1}^{p}t_i^{-\beta_i}\exp\big(- \sum_{i=1}^{p}N_i\cdot \log t_i  \big)v(t). 
\end{align}
	One can check that $\tilde{v}$ is single valued, and that $\nabla^{0,1}\tilde{v}=0$. We now fix a basis $v_1,\ldots,v_r$ of $V^\nabla$ such that each $v_i$ belongs to some $\bE_{\bm{\lambda}}$. Then the holomorphic sections $\tilde{v}_1,\ldots,\tilde{v}_r$ of $\cV$ defines     a prolongation of $\cV$ over $X$ which we denoted by $V^{\rm Del}_{\bm{\alpha}}$.  One can check that this construction does not depend on our choice of the basis. This is called the \emph{Deligne extension} of the flat bundle  $(V,\nabla)$ with the eigenvalues of the residue of $\nabla$ over $D_i$ lying in $[-\alpha_i,-\alpha_i+1)$.   Note that it is defined for any flat bundle $(V,\nabla)$ (not necessarily complex variation of Hodge structure).
	
If $(V,\nabla)$ underlies a complex polarized variation of Hodge structure $(V, \nabla, F^\bullet V, Q)$, we define
 $
	F^p_{\bm{\alpha}}:=j_*F^p\cap V_{\bm{\alpha}}^{\rm Del}.
	$ 
	It is called the extension of  the Hodge filtration.  It should be noted that, a priori, we do not know whether $F^p_{\bm{\alpha}}$ is locally free.

	\subsection{Acceptable bundles}
	\begin{dfn}(Admissible coordinate)\label{def:admissible}	Let $X$ be a complex manifold and let $D$ be a simple normal crossing divisor. Let $x$ be a point of $X$, and assume that $\{D_{j}\}_{ j=1,\ldots,\ell}$ 
	are the components of $D$ containing $p$. An \emph{admissible coordinate} around $x$
		is the tuple $(\Omega;z_1,\ldots,z_n;\varphi)$ (or simply  $(\Omega;z_1,\ldots,z_n)$ if no confusion arises) where
		\begin{itemize}
			\item $\Omega$ is an open subset of $X$ containing $x$.
			\item there is a holomorphic isomorphism   $\varphi:\Omega\to \Delta^n$ such that  $\varphi(D_j)=(z_j=0)$ for any
			$j=1,\ldots,\ell$.
		\end{itemize} 
		We shall write $\Omega^*:=\Omega-D$,   $\Omega(r):=\{z\in \Omega\mid |z_i|<r, \, \forall i=1,\ldots,n\}$ and $\Omega^*(r):=\Omega(r)\cap \Omega^*$.  
	\end{dfn}
	
	We define a (incomplete) Poincar\'e-type metric $\omega_P$ on $(\Delta^*)^\ell\times \Delta^{n-\ell}$  by
\begin{align}\label{eq: Poin}
	\omega_P=\sum_{j=1}^{\ell}\frac{\sqrt{-1}dz_j\wedge d\bar{z}_j}{|z_j|^2(\log |z_j|^2)^2}+\sum_{k=\ell+1}^{n} \sqrt{-1}dz_k\wedge d\bar{z}_k. 
\end{align}
	Note that 
	$$
	\omega_P=\hess \log \big(\prod_{j=1}^{\ell}(-\log |z_j|^2)^{-1}  \cdot \prod_{k=\ell+1}^{n}\exp(|z_k|^2)\big).
	$$
	
	For any system of Hodge bundles $(E,\theta,h)$, we have the following crucial norm estimate for its Higgs field $\theta$. The   one dimensional case  is due to Simpson \cite[Theorem 1]{Sim90} and the general case was proved by Mochizuki in \cite[Proposition 4.1]{Moc02}.  Its proof relies on a clever use of Ahlfors-Schwarz lemma.
	\begin{thm} \label{thm:moc}
		Let $(E,\theta,h)$ be a system of Hodge bundle  on $X\backslash D$.  Then  for any point $x\in D$, it has an admissible coordinate $(\Omega;z_1,\ldots,z_n)$  such that the norm 
	 $
		| \theta|_{h,\omega_P}\leq C  
		$ 
		holds over  $\Omega^*$	for some constant $C>0$. Here 	$| \theta|_{h,\omega_P}$ denotes the norm of $\theta$ with respect to $h$ and $\omega_P$.  \qed
	\end{thm}

Here we also recall the following definition in \cite[Definition 2.7]{Moc07}.
\begin{dfn}[Acceptable bundle]
	Let $(E, h)$  be a hermitian vector bundle over $X\backslash D$. We say that $(E, h)$ is an acceptable at $p\in D$, if the following holds: there is an  admissible coordinate $(\Omega;z_1,\ldots,z_n)$ around $p$, such that the norm  $\lvert R(E,h) \rvert_{h,\omega_P}\leq C$ for some $C>0$.  Here $R(E,h)$ is the Chern curvature of $(E,h)$.   When $(E_,h)$ is acceptable at any point $p$ of $D$, it is called acceptable.  
\end{dfn} 
Hodge filtrations and Hodge bundles endowed with the Hodge metric are all acceptable. 
\begin{lem}\label{lem:acceptable}
Let   $(V, \nabla, F^\bullet, Q)$  be a complex polarized variation of Hodge structure of weight $m$ on   $X\backslash D$. Let $h$ (resp. $h_{p,q}$) be the hermitian metric on $V$ (resp. on $E^{p,q}$) introduced in \cref{sec:VHS}.  Consider the induced hermitian metric $h_p:=h|_{F^p}$ on the Hodge filtration $F^p$. Then  both $(F^p, h_p)$ and $(E^{p,q},h_{p,q})$ are acceptable  bundles. 
\end{lem}
\begin{proof}
We write $\theta_{p,q}:=\theta|_{E^{p,q}}$ and let $\theta_{p,q}^\dagger:E^{p-1,q+1}\to A^{0,1}(E^{p,q})$ be its adjoint with respect to $h_{p,q}$.	For the hermitian bundle $(F^p, h_p)$, its curvature is
	$$	
 R_{h_p}(F^p)=- 2\theta_{m,0}^\dagger\wedge\theta_{m,0}+  2\sum_{i=1}^{m-p}(-\theta_{m-i,i}^\dagger\wedge \theta_{m-i,i}-\theta_{m-i+1,i-1}\wedge\theta_{m-i+1,i-1}^\dagger)+\theta_{p,m-p}^\dagger\wedge\theta_{p,m-p}.
	$$
	The curvature of the bundle $(E^{p,q},h_{p,q})$ is 
	$$
	 R_{h_{p,q}}(E_{p,q})=-\theta_{p,q}^\dagger\wedge\theta_{p,q}-\theta_{p+1,q-1}\wedge\theta_{p+1,q-1}^\dagger.
	$$
	By \cref{thm:moc}, for any point $x\in D$ there is an admissible coordinate $(\Omega;z_1,\ldots,z_n)$ around $x$  such that the norm 
	$$
\sum_{p+q=m}	| \theta_{p,q}|_{h_{p,q},\omega_P}=	| \theta|_{h,\omega_P}\leq C  
	$$
	holds over  $\Omega^*$	for some constant $C>0$.  Since $\theta_{p,q}^\dagger$ is the adjoint of $\theta_{p,q}$ with respect to $h_{p,q}$, one has $
	| \theta^\dagger_{p,q}|_{h,\omega_P}\leq C  
	$ for any $p$. It follows that $|R_{h_p}(F^p)|_{h_{p,q},\omega_P}\leq C'$ and $| R_{h_{p,q}}(E_{p,q})|_{h_{p,q},\omega_P}\leq C'$ for some $C'>0$. Hence  $(F^p, h_p)$ and $(E^{p,q},h_{p,q})$ are both acceptable.  
\end{proof}

\subsection{Adapted to log order}\label{sec:adapt}
We recall some notions in \cite[\S 2.2.2]{Moc07}.  
Let $X$ be a complex manifold, $D$ be a simple normal crossing divisor on $X$, and $E$ be a holomorphic vector bundle on $X\backslash D$ such that  $E|_{X\backslash D}$ is equipped with  a hermitian metric $h$. Let $\bm{v}=(v_1,\ldots,v_r)$ be a smooth frame of $E|_{X\backslash D}$. We obtain the $H(r)$-valued function $H(h,\bm{v})$ defined over $X\backslash D$,whose $(i,j)$-component is given by $h(v_i,v_j)$.  

Let us consider the case $X=\bD^n$, and $D=\sum_{i=1}^{\ell}D_i$ with $D_i=(z_i=0)$. We have the coordinate $(z_1,\ldots,z_n)$. Let  $h$, $E$ and $\bm{v}$  be as above.

\begin{dfn}\label{def:adapted}
	A smooth frame $\bm{v}$ on $X\backslash D$ is called \emph{adapted up to log order}, if  the following inequalities hold over $X\backslash D$:
	$$
	C^{-1}(-\sum_{i=1}^{\ell}\log |z_i|)^{-M}\leq H(h,\bm{v})\leq  C(-\sum_{i=1}^{\ell}\log |z_i|)^{M}
	$$   
	for some positive numbers $M$ and $C$.
\end{dfn}

\subsection{Parabolic vector bundles}\label{sec:parahiggs}

In this subsection, we recall the notions of parabolic (vector) bundles. For more details we refer to \cite{Moc06}. Let $X$ be a complex manifold, $D=\sum_{i=1}^{\ell}D_i$ be a reduced simple normal crossing divisor, $U=X\backslash D$ be the complement of $D$ and $j:U\to X$ be the inclusion.

\begin{dfn}\label{dfn:parab-higgs}
	A parabolic bundle $	\mathcal{P}_* E$ on $(X, D)$ is a  holomorphic vector bundle $E$ on $U$, together with an $\mathbb{R}^\ell$-indexed
	filtration $\pae$ ({\em parabolic structure}) by locally free subsheaves of  $j_*E$ such that
	\begin{thmlist}[leftmargin=0.7cm]
		\item $\bm{\alpha}\in \mathbb{R}^\ell$ and $\pae|_U=E$. 
		\item $\pa E\subset \cP_{\!\bm{\beta}}E$ if $\alpha_i\leq \beta_i$ for all $i$.
		\item  For $1\leq i\leq \ell$, $\mathcal{P}_{\bm{\alpha}+\bm{1}_i}E= \pae\otimes \cO_X(D_i)$, where $\bm{1}_i=(0,\ldots, 1,  \ldots, 0)$ with $1$ in the $i$-th component.
		\item \label{semiconti} $\mathcal{P}_{\bm{\alpha}+\bm{\ep}}E= \pae$ for any vector $\bm{\epsilon}=(\epsilon, \ldots, \epsilon)$ with $0<\epsilon\ll 1$.
		\item  The set of {\em weights}   $\{\bm{\alpha} \mid \pae/\mathcal{P}_{\! <\bm{\alpha}}E\} \not= 0$    is  discrete in $\mathbb{R}^\ell$.
	\end{thmlist}
\end{dfn}   


	\subsection{Prolongation via norm growth}\label{sec:pro}
	Let $X$ be a complex manifold, $D=\sum_{i=1}^{\ell}D_i$ be a   simple normal crossing divisor, $U=X\backslash D$ be the complement of $D$ and $j:U\to X$ be the inclusion. Let $(E,h)$ be a hermitian vector bundle on $U$.	For any $\bm{\alpha}=(a_1,\ldots,a_\ell)\in \mathbb{R}^\ell$, we can prolong $E$ over $X$ by a sheaf of $\cO_X$-module $\pae$  as follows: 
	\begin{align*} 
		\pa  E(U)=\{\sigma\in\Gamma(U\backslash D,E|_{U\backslash D})\mid |\sigma|_h\lesssim {\prod_{i=1}^{\ell}|z_i|^{-\alpha_i-\ep}}\  \ \mbox{for all}\ \ep>0\}. 
	\end{align*} 
 In \cite[Theorem 21.3.1]{Moc11} Mochizuki proved that the prolongation of acceptable bundles defined above are parabolic bundles. 
	\begin{thm}[Mochizuki]\label{thm:Moc2}
		Let $(E,h)$ be an acceptable bundle over $X\backslash D$. Then $	\mathcal{P}_* E$ defined above is a parabolic bundle. \qed
	\end{thm}
	
		\subsection{Period domain and period mapping}\label{sec:period}
	In this subsection we quickly review the definitions of period domain and period mapping.   We refer the readers to \cite{CMP17,KKM11,SS22} for more details.

Let $(V=\oplus_{p+q=m}V^{p,q},Q)$ be a polarized complex Hodge structure of weight $m$  defined in \cref{sec:VHS}.  Recall that the Hodge filtration is defined to be $F^p:=\oplus_{i\geq p}V^{i,m-i}$. After fixing  $m$ and $\dim_{\bC}F^p$, the set of all such filtration $F^\bullet$  is a complex flag manifold, which is denoted by $\check{\sD}$. It is a closed submanifold of a product of Grassmannians, and   is thus a projective manifold. The subset $\sD$ of all complex polarized Hodge structures are charcterized by
	\begin{enumerate}[label=(\alph*)]
		\item $
		F^p=F^p\cap (F^{p+1})^\perp\oplus F^{p+1}$. 
		\item  $(-1)^pQ$ is positive definite over $F^p\cap (F^{p+1})^\perp$.
	\end{enumerate}
	It is an open submanifold of $\check{\sD}$. We usually write $F$  instead of $F^\bullet$ to lighten the notation. Since the groups ${\rm GL}(V)$ and $G:=U(V,Q)$ act transitively on $\check{\sD}$ and $\sD$ respectively,  $\check{\sD}$ and $\sD$ are thus homogeneous spaces.
	
	For any Hodge structure $F\in \check{\sD}$, the holomorphic tangent space $T_{\check{\sD},F}$ of $\check{\sD}$  at $F$ is identified with  
	\begin{align*} 
	\End(V)/ \{A\in \End(V) \mid A(F^p)\subset F^p \    \mbox{for all}\  p \}.
	\end{align*}
	For any $A\in \End(V)$, we denote by $[A]_F$ its image in $T_{\check{\sD},F} $.  
	
	A tangent vector $[A]_F$  in $T_{\check{\sD},F}$ is called \emph{horizontal} if $A(F^p)\subset F^{p-1}$ for all $p$. The subbundle of $T_{\check{\sD}}$ consisting of horizontal vectors is denoted by $T_{\check{\sD}}^{-1,1}$ and one can show that   it is a holomorphic subbundle of $T_{\check{\sD}}$. A holomorphic map $f:\Omega\to \check{\sD}$ from a complex manifold $\Omega$ is called \emph{horizontal }if $df: T_\Omega\to f^*T_{\check{\sD}}$ factors through $f^*T_{\check{\sD}}^{-1,1}$.
	
	A complex (unpolarized) variation of Hodge structure  $(V=\oplus_{p+q=m},\nabla)$  over a complex manifold $\Omega$ induces a horizontal holomorphic map $\Phi: \tilde{\Omega}\to \check{\sD}$ by the Griffiths transversality, where $\tilde{\Omega}$  is the universal cover of $\Omega$.  Here we choose the reference space of $\check{\sD}$ to be  the space of multivalued flat sections $V^\nabla$.  $\Phi$ is called the period mapping associated to  $(V,\nabla,F^\bullet)$.   When this   complex variation of Hodge structure  is moreover polarized, $\Phi$ factors through $\sD$.

	\section{Nilpotent orbit theorem}
	\subsection{Two results of $L^2$-estimate}
	Set $X=\Delta^n$ and $D=(z_1\cdots z_\ell=0)$. We equip the complement $U:=X\backslash D$ with the Poincar\'e metric $\omega_P$ defined in \eqref{eq: Poin}. Write 
	\begin{align}\label{eq:Xr}
		 X(r):=\{z\in X\mid |z_i|<r\ \mbox{for}\ i=1,\ldots,\ell\} \quad \mbox{and}\quad U(r)=X(r)\cap U.
	\end{align} 
	\begin{lem}\label{thm:L2}
		Let $(F, h_F)$ be  a hermitian vector bundle on  $U$ such that $|R_{h_F}(F)|\leq C\omega_P$ for some constant $C>0$.  Then for any section $\eta\in \sC^\infty(U, \Lambda^{0,1}T_U^*\otimes F)$ such that $|\eta|_{h_F,\omega_P}\lesssim   \prod_{j=1}^{\ell}|z_j|^{\ep}$ for some  $\ep>0$, and $\db\eta=0$, there exists $\sigma\in  \sC^\infty(U, F)$ such that  $\db \sigma=\eta$ and
\begin{align*}
		\int_{U} |\sigma|^2_{h_F} \prod_{j=1}^{\ell}(-\log |z_j|^2)^{N}  d\mbox{vol}_{\omega_P}<\infty 
	\end{align*}
		for some $N\gg 1$.
	\end{lem}
	\begin{proof}
		For  the line bundle $K_U^{-1}$ endowed with the natural metric $g$ induced by $\omega_P$,  it is acceptable. 
		Hence for the hermitian vector  bundle $(E, h):=(K_U^{-1}\otimes F, g\cdot h_F)$,  it is also acceptable.   
It follows from  \cite[Lemma 1.10]{DH19} that one can  choose $N\gg 1$ such that 
		$$ i R_h(E)\geq_{Nak} -(N-1)\omega_P\otimes \ide_E,$$
		where ``$\geq_{Nak}$" stands for Nakano semipositive (see \cite[D\'efinition 2.2]{Dem82}).  	For the function	\begin{align}\label{eq:potent}
		 \varphi:=\log \big(\prod_{j=1}^{\ell}(-\log |z_j|^2)^{-1}  \cdot \prod_{k=\ell+1}^{n}\exp(|z_k|^2)\big),
		\end{align}
	one has
		 $
		\hess \log\varphi=\omega_P.
		$ 
		For any $k\in \bZ$	we define a new metric $h(k)=h\cdot e^{-k\varphi}$ for $E$. Therefore, 
		$$
		iR_{h(N)}(E)=iR_h(E)+N\omega_P\otimes \ide_E \geq_{Nak} \omega_P\otimes \ide_E.
		$$ Note that  $\sC^\infty(U, \Lambda^{n,1}T_U^*\otimes E)=\sC^\infty(U, \Lambda^{0,1}T_U^*\otimes F)$ with $|\eta|_{h,\omega_P}=|\eta|_{h_F,\omega_P}$. Since $|\eta|_{h_F,\omega_P}\lesssim   \prod_{j=1}^{\ell}|z_j|^{\ep}$,  
		$|\eta|_{h(N),\omega_P}\leq C'$ for some $C'>0$.  Hence $\lVert \eta \rVert_{h(N), \omega_P}<\infty$. Thanks to  the Demailly-H\"ormander $L^2$-estimate \cite[Th\'eor\`eme 4.1 and Remarque 4.2]{Dem82},  there exists $   \sigma\in  \sC^\infty(U, K_U\otimes E)= \sC^\infty(U, F)$  such that
		 $
		\db \sigma=\eta
		$ 
		and $\lVert \sigma\rVert_{h(N)}<\infty$.   
Here we note that the smoothness of $\sigma$ follows from the elliptic regularity of the Laplacian.  
The lemma is proved.
	 \end{proof}
	
	\begin{lem}\label{lem:estimate}
	Let $(E, h)$ be  a  hermitian vector bundle on  $U$ such that $|R_{h}(E)|\leq C\omega_P$ for some constant $C>0$.  Assume that $\sigma \in H^0(U,E)$ such that $\lVert\sigma\rVert_{h(N)}<\infty$ for some integer $N\geq 1$,  where $h(N):=h\cdot e^{-N\varphi}$ with $\varphi$ defined in \eqref{eq:potent}, then  over $U(\frac{1}{2})$, $|\sigma|_{h}\lesssim \prod_{j=1}^{\ell}|z_j|^{-\ep}$ for any $\ep>0$ .
		\end{lem}
	\begin{proof}
Since $|R_{h}(E)|\leq C\omega_P$ for some constant $C>0$, it follows from  \cite[Lemma 1.10]{DH19}   that  $(E,h(-N'))$ is Griffiths semi-negative for some $N'\gg 1$, where $h(-N'):=h\cdot e^{N'\varphi}$ with $\varphi$ defined in \eqref{eq:potent}.  One can show that   $\log |\sigma|^2_{h(-N')}$ is a plurisubharmonic function.  
		For any $z\in U^*(\frac{1}{2})$, one has  
		\begin{align*}
			\log |\sigma(z)|^2_{h(-N')}&\leqslant  \frac{4^n}{\pi^n \prod_{i=1}^{\ell}|z_i|^2}\int_{\Omega_z}  \log |\sigma(w)|^2_{h(-N')}d\mbox{vol}_{g}\\
			&\leqslant \log \big(\frac{4^n}{\pi^n \prod_{i=1}^{\ell}|z_i|^2}\cdot \int_{\Omega_z}   |\sigma(w)|^2_{h(-N')}d\mbox{vol}_{g}\big)\\
			&\leqslant \log \big(C \int_{\Omega_z}    \frac{1}{  \prod_{i=1}^{\ell}|w_i|^2}|\sigma(w)|^2_{h(-N')}d\mbox{vol}_{g}\big)\\
			&\leqslant C_1+ \log \int_{\Omega_z}   |\sigma(w)|^2_{h(-N')}\cdot |\prod_{i=1}^{\ell}(\log |w_i|^2)^2|  d\mbox{vol}_{\omega_P}\\
			& \leqslant  C_2+\log   \int_{\Omega_z}   |\sigma(w)|^2_{h(N)}d\mbox{vol}_{\omega_P}\\
			&\leqslant C_2+\log \lVert \sigma\rVert_{h(N)}^2 
		\end{align*}
		where $\Omega_z:=\{w\in U^*\mid |w_i-z_i|\leq \frac{|z_i|}{2} \mbox{ for } i\leq \ell;   |w_i-z_i|\leq \frac{1}{2} \mbox{ for } i>\ell\}$ and $g$ is the Euclidean metric.  $C_1, C_2 $ are two positive constants which do not depend on $z\in U^*(\frac{1}{2})$. The first inequality is due to the   mean value inequality, and the second one follows from the Jensen inequality.    It follows that
		\begin{align*}
			|\sigma(z)|_{h}&=|\sigma(z)|_{h(-N')}\cdot  \big(\prod_{j=1}^{\ell}(-\log |z_j|^2)^{\frac{N'}{2}}  \cdot \prod_{k=\ell+1}^{n}\exp(|z_k|^2)^{-\frac{N'}{2}}\big) \\ 
			&\leq \exp(\frac{C_2}{2})\cdot \lVert \sigma\rVert_{h(N)}\cdot   \big(\prod_{j=1}^{\ell}(-\log |z_j|^2)^{\frac{N'}{2}}  \cdot \prod_{k=\ell+1}^{n}\exp(|z_k|^2)^{-\frac{N'}{2}}\big)\lesssim  \prod_{i=1}^{\ell}|z_i|^{-\ep}
		\end{align*}
		for any $\ep>0$. 
	\end{proof}  
	\subsection{Proof of \Cref{main}}
	We first prove that the Deligne extension of the flat bundle unerlying a complex variation of Hodge structure coincides with the prolongation defined in \cref{sec:pro}.
	 \begin{proposition}\label{lem:equal}
Let $X$ be a complex manifold, $D=\sum_{i=1}^{\ell}D_i$ be a   simple normal crossing divisor. For a complex variation of Hodge structure $(V,\nabla,F^\bullet,h)$ defined on $U:=X\backslash D$, one has $V_{\bm{\alpha}}^{\rm Del}=\mathcal{P}_{\! \bm{\alpha}}V$ for any multi-index $\bm{\alpha}\in \bR^\ell$, where $\mathcal{P}_{\! \bm{\alpha}}V$  is the prologation of $V$ defined in \cref{sec:pro}. 
	 \end{proposition}
 \begin{proof}
\noindent {\it Step 1: we prove that $V_{\bm{\alpha}}^{\rm Del}\subset \mathcal{P}_{\! \bm{\alpha}}V$. }  We will use the notation in \cref{sec:Del}.	Since this is a local problem, we can assume that $X=\Delta^n$ and $D=(t_1\cdots t_p=0)$.  By the construction of  $V_{\bm{\alpha}}^{\rm Del}$ one can take a basis $v_1,\ldots,v_r$ of $V^\nabla$ with each $v_i\in \bE_{\bm{\lambda}(v_i)}$ for some $\bm{\lambda}(v_i)\in Sp$ such that  $\{\tilde{v}_1,\ldots,\tilde{v}_r\}$ defined in \eqref{eq:twist}  forms a basis of $V_{\bm{\alpha}}^{\rm Del}$.  It thus suffices to estimate the norm
 \begin{align*} 
	\tilde{v}(t):=\exp\big(- \sum_{i=1}^{p}( \beta_iI+N_i)\cdot \log t_i  \big)v(t) = \prod_{i=1}^{p}t_i^{-\beta_i}\exp\big(- \sum_{i=1}^{p}N_i\cdot \log t_i  \big)v(t)
\end{align*}
for any $\bm{\lambda}$ and $v\in \bE_{\bm{\lambda}}$.

	By the weaker norm estimate in \cite[Lemma 9.31]{Moc07} for general harmonic bundles, there exists a frame   $v_1,\ldots,v_r$ of $V^\nabla$ with each $v_i\in \bE_{\bm{\lambda}(v_i)}$   and $\{a_{ij}\}_{i=1,\ldots,r; j=1,\ldots,p}\subset  \bR$  such that if we put 
$
 {v}'_i:=v_i\cdot \prod_{j=1}^{p}|t_j|^{-a_{ij}},
$ 
then for the multivalued smooth sections  ${\bm{v}}'=({v}'_1,\ldots, {v}'_r)$,  over a given  sector of  $U$  one has the   norm estimate
\begin{align}\label{eq:norm}
	 (\prod_{i=1}^p|\log |t_i||)^{-M}\lesssim H(h, {\bm{v}'}) \lesssim (\prod_{i=1}^p|\log |t_i||)^M
\end{align}  
for some $M>0$.  Here  $h$ is the Hodge metric and   $H(h, {\bm{v}'}) $ is the $H(r)$-valued function   defined in \cref{sec:adapt}. 

 Fix some  $\{t_1,\ldots,t_{p}\}\subset  \Delta^*$. For each $\ell=1,\ldots,p$,   over any sector of $\Delta^*$,  we have
$$  |\log |t||)^{-M}\lesssim
|v_i|^2(t_1,\ldots,t_{\ell-1},t,t_{\ell+1},\ldots,t_n) \lesssim  |\log |t||^M$$ 
by the Hodge norm estimate in one dimensional case in \cite{SS22,Sim90}. Together with \eqref{eq:norm}, it implies that $a_{ij}=0$ for all ${i=1,\ldots,r; j=1,\ldots,\ell}$.  Therfore, we conclude that over a given  sector of  $U$  one has the   norm estimate
\begin{align} 
	(\prod_{i=1}^p|\log |t_i||)^{-M}\lesssim H(h,{\bm{v}}) \lesssim (\prod_{i=1}^p|\log |t_i||)^M.
\end{align}  
Then for any multivalued flat section $v\in V^\nabla$,  over any given  sector of  $U$ one has  
$$ (\prod_{i=1}^p|\log |t_i||)^{-M}\lesssim|v(t)|_h\lesssim (\prod_{i=1}^p|\log |t_i||)^M$$
for some $M>0$.  Since all $N_i$ are nilpotent and pairwise commute, 
 $$
\exp\big(- \sum_{i=1}^{p}N_i\cdot \log t_i  \big)v(t)=\sum_{i=1}^{p}\sum_{k=0}^{N}\frac{1}{k!}(\log t_i)^k(N_i^k v)(t)
$$ 
for some integer $N>0$. Note that if $v\in \bE_{\bm{\lambda}}$, we also have $N_i^k v\in \bE_{\bm{\lambda}} $ for any $k\geq 0$.   Since one can cover $X\backslash D$ by finitely many sectors,   this proves the norm estimate
$$
 |\exp\big(- \sum_{i=1}^{p}N_i\cdot \log t_i  \big)v(t)|_h\lesssim (\prod_{i=1}^p|\log |t_i||)^{M'} 
$$
for some $M'>0$. 
Hence $$|\tilde{v}(t)|_h  \lesssim \prod_{i=1}^{p}|t_i|^{-\beta_i-\ep}\lesssim\prod_{i=1}^{p}|t_i|^{-\alpha_i-\ep}$$ for any $\ep>0.$  This proves the inclusion $V_{\bm{\alpha}}^{\rm Del}\subset \mathcal{P}_{\! \bm{\alpha}}V$ by the   definition of $\pa V$ in \cref{sec:pro}.

\medspace

\noindent {\it Step 2: we prove that  $\mathcal{P}_{\! \bm{\alpha}}V\subset V_{\bm{\alpha}}^{\rm Del}$.}   First we note that the decomposition $V^\nabla=\oplus_{\bm{\lambda}\in Sp}\bE_{\bm{\lambda}}$  induces a decomposition of the flat bundle $(V,\nabla)$ into 
\begin{align} \label{eq:decom}
(V,\nabla)=\oplus_{\bm{\lambda}\in Sp}(V(\bm{\lambda}), \nabla|_{V(\bm{\lambda})}),
\end{align}
where $(V(\bm{\lambda}), \nabla|_{V(\bm{\lambda})})$ is the flat subbundle induced by $\bE_{\bm{\lambda}}$. 
We fix a basis $(v_1,\ldots,v_r)\in V^\nabla$ such that  $v_i\in \bE_{\bm{\lambda}(v_i)}$ for some $\bm{\lambda}(v_i) \in Sp$. This means that  such basis is compatible with the above decomposition \eqref{eq:decom}; namely $v_j$ is a mutivalued flat section of $(V(\bm{\lambda}(v_j)), \nabla|_{V(\bm{\lambda}(v_j))})$.  Consider the dual  bundle $V^*$ of $V$,  and it can endowed with the natural connection $\nabla$   defined by
$$
(\nabla\mu)v=d(\mu(v))-\mu(\nabla(v))
$$
for $\mu$ and $v$  sections in $V^*$ and $V$ respectively. $(V^*,\nabla)$ is thus also a  flat bundle.  Moreover,   $(V^*)^\nabla$ is the dual space of $(V^\nabla)$.   Consider the dual basis $(v_1^*,\ldots,v^*_r)$ of $(v_1,\ldots,v_r)$. Since
$$
(\nabla v_i^*)v_j=d(v_i^*(v_j))-v_i^*(\nabla v_j)=0,
$$
one has 
$v_i^*\in (V^*)^\nabla$.  Recall that  $T_j$ is the monodromy transformation of $(V,\nabla)$ with respect to $\gamma_j$ defined by  
$$
v(t_1,\ldots,e^{2\pi i}t_j,\ldots,t_{p+q})=(T_j v)(t_1,\ldots,t_{p+q})
$$
for any $v\in V^\nabla$. Let us denote by $\tilde{T}_j$ the monodromy transformation of $(V,\nabla)$ with respect to $\gamma_j$ defined by  
$$
\mu(t_1,\ldots,e^{2\pi i}t_j,\ldots,t_{p+q})=(\tilde{T}_j \mu)(t_1,\ldots,t_{p+q})
$$
for any $\mu\in (V^*)^\nabla$. Then  for any $v\in V^\nabla$ and any $\mu\in (V^*)^\nabla$ one has
\begin{align*}
\mu(t)(v(t))&=\mu(t_1,\ldots,e^{2\pi i}t_j,\ldots,t_{p+q})(v(t_1,\ldots,e^{2\pi i}t_j,\ldots,t_{p+q}))\\
&=(\tilde{T}_j\mu(t))(T_jv(t))=(\tilde{T}_j\mu)(T_j v)=(T_j^*\tilde{T}_j \mu)(v),
\end{align*}
where $T_j^*:(V^*)^\nabla\to (V^*)^\nabla$ is the adjoint of $T_j$.
Hence 
\begin{align}\label{eq:mono}
\tilde{T}_j=(T_j^*)^{-1}. 
\end{align}
It follows that $Sp(\tilde{T}_j)=\{\lambda^{-1} \}_{\lambda\in Sp(T_i)}$. Set $\bE(\tilde{T}_j,\lambda_j)\subset (V^*)^\nabla$ to be the corresponding   eigenspace of $\lambda_j\in Sp(\tilde{T}_j)$.  We know that all $\lambda_j\in Sp(\tilde{T}_j)$ have norm $1$ since   $(V^*,\nabla)$ admits a complex variation of Hodge structure.  For $\bm{\lambda}=(\lambda_1,\ldots,\lambda_p)\in Sp$, we define
$$
\tilde{\bE}_{\bm{\lambda}}:=\cap_{j=1}^{p}\bE(\tilde{T}_j,\lambda^{-1}_j)\subset (V^*)^\nabla
$$
Since $T_j's$ are pairwise commute,  one has
$$
(V^*)^\nabla=\oplus_{\bm{\lambda}\in Sp} \tilde{\bE}_{\bm{\lambda}},
$$
and $\tilde{E}_{\bm{\lambda}}$ is an invariant subspace of $\tilde{T}_j$ for any $\bm{\lambda}\in Sp$ and any $j$.
 
By  \cref{lem:linear} below, one can show that for any $\mu\in \tilde{\bE}_{\bm{\lambda'}}$ and $v\in \bE_{\bm{\lambda}}$, $\mu(v)=0$ if  $\bm{\lambda}\neq \bm{\lambda}'$, which implies that 
 $
 v_j^*\in  \tilde{\bE}_{\bm{\lambda}(v_j)}.
 $


For $\bm{\lambda}\in Sp$, there exists a unique $\beta_i\in (\alpha_i-1, \alpha_i]$ such that $\exp(2\pi i \beta_i)=\lambda_i$.  Denote  $N_i:=\frac{\log(	\lambda_i^{-1}T_i|_{ \bE_{\bm{\lambda}}})}{2\pi i}$. Recall that  for any $v\in \bE_{\bm{\lambda}}$,  we define
\begin{align*} 
	\tilde{v}(t):=\exp\big(- \sum_{i=1}^{p}( \beta_iI+N_i)\cdot \log t_i  \big)v(t) = \prod_{i=1}^{p}t_i^{-\beta_i}\exp\big(- \sum_{i=1}^{p}N_i\cdot \log t_i  \big)v(t). 
\end{align*} 
 Likewise, since  $\lambda_i\tilde{T}_i|_{ \tilde{\bE}_{\bm{\lambda}}}$ is unipotent, its logarithm can be defined as
$$
\log(	\lambda_i \tilde{T}_i|_{ \tilde{\bE}_{\bm{\lambda}}}):=\sum_{k=1}^{\infty}(-1)^{k+1}\frac{(\lambda_i\tilde{T}_i|_{ \tilde{\bE}_{\bm{\lambda}}}-I)^k}{k}.
$$
Write  $\tilde{N}_i:=\frac{\log(	\lambda_i \tilde{T}_i|_{ \tilde{\bE}_{\bm{\lambda}}})}{2\pi i}$.  Then for any $\mu\in\tilde{ \bE}_{\bm{\lambda}}$, we define 
\begin{align} \label{eq:defg}
	\tilde{\mu}(t):=\exp\big(- \sum_{i=1}^{p}( -\beta_iI+\tilde{N}_i)\cdot \log t_i  \big)\mu(t) = \prod_{i=1}^{p}t_i^{\beta_i}\exp\big(- \sum_{i=1}^{p}\tilde{N}_i\cdot \log t_i  \big)\mu(t). 
\end{align}
Since $\tilde{T}_i=(T_i^*)^{-1}$, one has
 $
\tilde{N}_j=-N_j^* 
$. Therefore,
\begin{align*} 
	\tilde{\mu}(t)(\tilde{v}(t))&=\exp\big(- \sum_{i=1}^{p}\tilde{N}_i\cdot \log t_i  \big)\mu(t)\Big(\exp\big(- \sum_{i=1}^{p}N_i\cdot \log t_i  \big)v(t)\Big) \\
	&=\exp\big(\sum_{i=1}^{p} {N}^*_i\cdot \log t_i  \big)\mu(t)\Big(\exp\big(- \sum_{i=1}^{p}N_i\cdot \log t_i  \big)v(t)\Big) \\
	&= \mu(v)=\mbox{constant}.
\end{align*}
This implies that  $\tilde{v}_i^*(t)(\tilde{v}_j(t))=v_i^*(v_j)\equiv \delta_{ij}$ if $\bm{\lambda}=\bm{\lambda}'$, where $\tilde{v}^*_i$ is defined in \eqref{eq:defg} in terms of $
v_i^*\in  \tilde{\bE}_{\bm{\lambda}(v_i)}. $  

If $\mu\in\tilde{ \bE}_{\bm{\lambda}}$ and $v\in \bE_{\bm{\lambda}'}$ with $\bm{\lambda}\neq \bm{\lambda}'$, the above construction shows that
 $
\tilde{\mu}
$ and $\tilde{v}$ are holomorphic sections of $V^*(\bm{\lambda})$ and $V(\bm{\lambda}')$. Here $V(\bm{\lambda}')$ is the invariant flat subbundle of $(V,\nabla)$ defined in \eqref{eq:decom}, and $V^*(\bm{\lambda})$  is defined to be the invariant flat subbundle of $(V^*,\nabla)$ generated by $\tilde{E}_{\bm{\lambda}}$.  
Hence $\tilde{v}^*_i$ and $\tilde{v}_j$ are holomorphic sections of $V^*(\bm{\lambda}(v_i))$ and $V(\bm{\lambda}(v_j))$ respectively.  This shows that
$\tilde{v}_i^*(t)(\tilde{v}_j(t))\equiv 0$ for $\bm{\lambda}\neq \bm{\lambda}'$ by \cref{lem:linear}.

In conclusion, we prove that $\tilde{v}^*_1,\ldots, \tilde{v}^*_r$ is the dual frame of $\tilde{v}_1,\ldots,\tilde{v}_r$.   

Recall that    $v_j\in \bE_{\bm{\lambda}(v_j)}$ for some $\bm{\lambda}(v_j)\in Sp$.  There exists a unique $\beta(v_j)_i\in (\alpha_i-1, \alpha_i]$ such that $\exp(2\pi i \beta(v_j)_i)=\bm{\lambda}(v_j)_i$.  Define a smooth section $\tilde{v}_j'=\tilde{v}_j\cdot \prod_{i=1}^p |t_i|^{\beta(v_j)_i}$. By the norm estimate in the first step, for all $\tilde{v}_j'$ one has the norm estimate
$$
|\tilde{v}_j'|_h\lesssim (\prod_{i=1}^p|\log |t_i||)^{M} 
$$
for some $M>0$. It follows that  
$$
  H(h; \tilde{v}'_1,\ldots,\tilde{v}_r')\lesssim (\prod_{i=1}^p|\log |t_i||)^{M} 
$$
Here $H(h; \tilde{v}_1',\ldots, \tilde{v}_r')$ is a $r\times r$-matrix function whose $(i,j)$-component is $h(\tilde{v}_i',\tilde{v}_j')$.   On the other hand,  we put  $\mu'_i=\tilde{v}^*_i\cdot \prod_{j=1}^p |t_j|^{-\beta(v_i)_j}$.  Since  complex polarized variation of Hodge structure is functorial by taking dual, $(V^*,\nabla)$ admits a complex polarized variation of Hodge structure whose Hodge metric is the dual metric $h^*$ of the Hodge metric $h$ for $(V,\nabla,F^\bullet,Q)$. In the same manner we obtain
$$
|\mu_i'|_{h^*}\lesssim (\prod_{i=1}^p|\log |t_i||)^{M'} 
$$
for every $\mu_i'$ and some $M'>0$.  This implies that 
$$
H(h^*; \mu'_1,\ldots,\mu'_r)\lesssim (\prod_{i=1}^p|\log |t_i||)^{M'} 
$$
By our construction, $\mu'_1,\ldots,\mu'_r$ is the dual of the smooth frame $\tilde{v}'_1,\ldots,\tilde{v}'_r$. It follows that
$$
(\prod_{i=1}^p|\log |t_i||)^{-M'}\lesssim H(h^*; \mu'_1,\ldots,\mu'_r)^{-1}= H(h; \tilde{v}'_1,\ldots,\tilde{v}_r').
$$ Hence 
\begin{align}\label{eq:double}
(\prod_{i=1}^p|\log |t_i||)^{-M'}\lesssim H(h; \tilde{v}'_1,\ldots,\tilde{v}_r')\lesssim (\prod_{i=1}^p|\log |t_i||)^{M}. 
\end{align}
 Now we are ready to prove the inclusion  $\mathcal{P}_{\! \bm{\alpha}}V\subset V_{\bm{\alpha}}^{\rm Del}$. For any $s\in \pa V(U)$,   it can be written as
  $
 s=\sum_{i=1}^rf_i \tilde{v}_i
 $ 
 where $f_i$ is a holomorphic function on $U$. By \eqref{eq:double} one has
 $$
  \sum_{i=1}^r |f_i|^2\cdot \prod_{j=1}^{p}|t_j|^{-2\beta(v_i)_j} \cdot (\prod_{k=1}^p|\log |t_k||)^{-2M'}\lesssim |s|^2_h\lesssim\prod_{i=1}^p|t_j|^{-2\alpha_j-\ep}
 $$
 for any $\ep>0$. 
Since $\beta(v_i)_j\in (\alpha_j-1,\alpha_j]$, one can choose $\delta>0$ such that $\beta(v_i)_j-\alpha+1> \delta$ for all $v_i$ and every $j=1,\ldots,p$. The above inequality implies that   for every $f_i$, 
 $$
 |f_i|\lesssim  \prod_{i=1}^p|t_j|^{-1+\delta}.
 $$
 Hence all $f_i$ extend to   holomorphic functions over $X$. This proves that $s\in V_{\bm{\alpha}}^{\rm Del}(X)$  since $\tilde{v}_1,\ldots, \tilde{v}_r$ is a holomorphic basis of $V_{\bm{\alpha}}^{\rm Del}$ by the definition of Deligne extension in \cref{sec:Del}.  The inclusion $\mathcal{P}_{\! \bm{\alpha}}V\subset V_{\bm{\alpha}}^{\rm Del}$ is proved. We complete the proof of the proposition. 
   \end{proof}
We leave the proof of the following lemma of linear algebra to the reader. 
\begin{lem}\label{lem:linear}
	Let $T:V\to V$ be an isomorphism of a finite dimensional $\bC$-vector space $V$. Decompose $V=V_{\lambda_1}\oplus \ldots\oplus V_{\lambda_k}$ into  eigenspaces of $T$, where $\lambda_i$ is a   eigenvalue of $T$ and $V_{\lambda_i}$ is the corresponding  eigenspace. Denote by $V^*$ the dual vector space. Then for the isomorphism $(T^*)^{-1}:V^*\to V^*$, its  eigenvalues are $\lambda_1^{-1},\ldots, \lambda_k^{-1}$ and its  eigenspace decomposition is $V^*=V^*_{\lambda^{-1}_1}\oplus \ldots\oplus V^*_{\lambda^{-1}_k}$, where $V^*_{\lambda^{-1}_j}$ is the corresponding  eigenspace of $\lambda_j^{-1}$. Moreover, one has
	$\mu(v)=0$ if $\mu\in  V^*_{\lambda^{-1}_i}$ and $v\in V_{\lambda_j}$ with $i\neq j$. \qed
\end{lem}
 
 \begin{thm}\label{thm:exact}
	Let $X$ be a complex manifold and let $D=\sum_{i=1}^\ell D_i$ be a simple normal crossing divisor on $X$.  Let $(V, \nabla, F^\bullet, Q)$ be a complex polarized variation of Hodge structure  of weight $m$ on $X\backslash D$. 	Let $\cP_{*}F^p$ and $\cP_{*} E_{p,m-p}$ be the  induced filtered bundle of hermitian bundles $(F^p, h_p)$ and $(E_{p,m-p}, h_{p,m-p})$ defined in \cref{sec:pro}. Then  for every $p$,  $\cP_{*}F^p$ and $\cP_{*} E_{p,m-p}$  are parabolic bundles and for each multi-index $\alpha\in \bR^\ell$, there is a  natural  exact sequence
	\begin{align}\label{eq:exact}
0\to \mathcal{P}_{\! \bm{\alpha}}F^{p+1}\to  \mathcal{P}_{\! \bm{\alpha}}F^p\to \pa E_{p,m-p}\to 0.
	\end{align}
\end{thm}
\begin{proof}	
By \cref{lem:acceptable}, $(F^p, h_p)$ and $(E_{p,m-p}, h_{p,m-p})$ are acceptable bundles for every $p$. 	It follows from \cref{thm:Moc2} that $\cP_{*}F^p$ and $\cP_{*} E_{p,m-p}$    defined in \cref{sec:pro} are parabolic ones.  
	
We first show that	we can define \begin{align} \label{eq:exact2}
	0\to \mathcal{P}_{\! \bm{\alpha}}F^{p+1}\to  \mathcal{P}_{\! \bm{\alpha}}F^p\stackrel{q}{\to}\pa E_{p,m-p},
\end{align}
which is exact.  Note that we have the following exact sequence 
	\begin{align} \label{eq:exact3}
	0\to F^{p+1}\to  F^p\stackrel{q}{\to} E_{p,m-p}\to 0
\end{align}
on $X\backslash D$  by the definition of $\bC$-VHS.  Pick any $x\in D$ and any admissible coordinate $(\Omega;z_1,\ldots,z_n)$ centering at $x$ such that $D\cap \Omega=(z_1\cdots z_k=0)$. By the   prolongation via norm growth defined in \cref{sec:pro}, any section $s\in \pa F^{p+1}(\Omega)$ satisfies that $s\in F^{p+1}(\Omega\backslash D)$ and $|s|_{h^{p+1}}\lesssim \prod_{i=1}^{k}|z_i|^{-\alpha_i-\ep}$ for any $\ep>0$.  Since $h_{p+1}$ is the induced metric of $h_p$ on $F^{p+1}$, it follows that  
\begin{align}\label{eq:same norm}
	 |s|_{h_{p+1}}=|s|_{h_p} \lesssim \prod_{i=1}^{k}|z_i|^{-\alpha_i-\ep} 
\end{align}
 for any $\ep>0$.  Hence the inclusion $F^{p+1}\subset F^p$ also results in the inclusion  $\pa F^{p+1}\subset \pa F^{p}$.  We proved that \eqref{eq:exact2} is exact in the left. 

Note that the metric $h_{p,m-p}$ on $E_{p,m-p}$ is the quotient metric of $h_p$. It follows that  for any section $s\in \pa F^{p}(\Omega)$ , we have
$$
|q(s)|_{h_{p,m-p}}\leq |s|_{h_p}\lesssim \prod_{i=1}^{k}|z_i|^{-\alpha_i-\ep}$$ for any $\ep>0$.   
Hence the quotient $q:F^p\to E_{p,m-p}$  induces  the morphism  $\mathcal{P}_{\! \bm{\alpha}}F^p\to\pa E_{p,m-p}$ and thus \eqref{eq:exact2} can be defined.  
 Next we will show that  \eqref{eq:exact2} is exact in the middle. 
 
 Take any section $s\in \pa F^{p}(\Omega)$ such that $q(s)=0$.  Thanks to the exactness in  \eqref{eq:exact3}, we have $s\in F^{p+1}(\Omega\backslash D)$. By \eqref{eq:same norm}, we conclude that $s\in \pa F^{p+1}(\Omega)$. This implies  the exactness of  \eqref{eq:exact2}.

\medspace

 In what follows we will prove that \eqref{eq:exact2} is exact on the right. It suffices to prove that for any point $x\in D$ and any section $s\in  \pa E_{p,m-p}(\Omega)$ where $\Omega$ is a neighborhood of $x$, there is a section $\tilde{s}\in  \mathcal{P}_{\! \bm{\alpha}}F^p(\Omega')$ for some smaller neighborhood $\Omega'$ of $x$ such that $q(\tilde{s})=s|_{\Omega'}$. We shall construct such $\tilde{s}$ by utilizing the previous results on $L^2$-estimate in   \cref{thm:L2,lem:estimate}.
	
	Since this is a local problem, we can assume that $X=\Delta^n$, $D=(z_1\cdots z_\ell=0)$ and $x$ is the origin. We equip the complement $U:=X\backslash D$   with the Poincar\'e metric   $\omega_P$. Let $X(r)$ and $U(r)$ be defined as in \eqref{eq:Xr}. By the semicontinuity of the parabolic bundle in \cref{semiconti}, we can  choose $\bm{\beta}\in \bR^\ell$ such that $\beta_i>\alpha_i$ and 
\begin{align} \label{continuous}
 \mathcal{P}_{\! \bm{\beta}}F^{p}=\pa F^{p}.
\end{align}
Pick any $s\in \pa E_{p,m-p}(X)$. Then $s\in H^0(U, E_{p,m-p})$ with $|s|_{h_{p,m-p}}\lesssim \prod_{i=1}^{\ell}|z_i|^{-\alpha_i-\ep}$ for any $\ep>0$. We will construct a section  $\tilde{s}\in H^0(U(r), F^p)$ for some $0<r<1$ such that $q(\tilde{s})=s|_{U(r)}$ and $|\tilde{s}|_{h_p}\lesssim \prod_{i=1}^{\ell}|z_i|^{-\beta_i-\ep}$ for any $\ep>0$.  Note that there is a  canonical smooth  isomorphism (and isometry)
$$
\Phi: (F^p, h_p)\to (F^{p+1}, h_{p+1})\oplus (E_{p,m-p}, h_{p,m-p})
$$
such that   the holomorphic structure of $F^p$ via $\Phi$ is defined by 
$$
\begin{bmatrix}
\db_{F^{p+1}} & \theta_{p+1,m-p-1}^\dagger  \\
	0 &  \db_{E_{p,m-p}} 
\end{bmatrix},
$$
where $ \theta_{p+1,m-p-1}^\dagger$ is the adjoint of $\theta_{p+1,m-p-1}$ with respect to $h_{p+1,m-p-1}$. 
If $q(\tilde{s})=s$, then $\Phi(\tilde{s})=[\sigma, s]$ for some $\sigma\in \sC^\infty(U, F^{p+1})$ such that
$$
\begin{bmatrix}
	\db_{F^{p+1}} & \theta_{p+1,m-p-1}^\dagger  \\ 
	0 &  \db_{E_{p,m-p}} 
\end{bmatrix} \begin{bmatrix}
\sigma   \\ 
s 
\end{bmatrix}=0
$$
	 Hence $\db_{F^{p+1}}\sigma=-\theta_{p+1,m-p-1}^\dagger s$. We will  solve this $\db$-equation with proper norm estimate.

By \cref{thm:moc},  after we  replace $U$ by $U(r)$ for some $0<r<1$, we have $|\theta_{p+1,m-p-1}|_{h,\omega_P}\leq C$ over $U$.  This implies that  $|\theta_{p+1,m-p-1}^\dagger|_{h,\omega_P}\leq C$ over $U$. Hence, 
	 \begin{align*}
	 	|\theta_{p+1,m-p-1}^\dagger s|_{h_{p+1},\omega_P}  \leq 	|\theta_{p+1,m-p-1}^\dagger|_{h, \omega_P}\cdot  |s|_{h_{p,m-p}} \lesssim \prod_{i=1}^{\ell}|z_i|^{-\alpha_i-\ep}
	 \end{align*}
	  for any $\ep>0$.  We now introduce a new metric  for $F^{p}$ defined by 
	  $$h_{p}(\bm{\beta}):=h_{p}\cdot \prod_{i=1}^{\ell}|z_i|^{\beta_i}.$$ 
Since $\beta_i>\alpha_i$ for each $i$, we have  
	$$  |\theta_{p+1,m-p-1}^\dagger s|_{h_{p+1}(\bm{\beta}),\omega_P}  \lesssim   \prod_{j=1}^{\ell}|z_j|^{\delta}$$ for some $\delta>0$.
	  Note that $\theta_{p+1,m-p-1}^\dagger s\in A^{0,1}(E_{p+1,m-p-1})$. We have
\begin{align*}
	\db_{F^{p+1}}(\theta_{p+1,m-p-1}^\dagger s)&=(\db_{E_{p+1,m-p-1}}+\theta_{p,m-p}^\dagger)(\theta_{p+1,m-p-1}^\dagger s)\\
&	= \db_{E_{p+1,m-p-1}} (\theta_{p+1,m-p-1}^\dagger s)\\
&=	(D_{h}^{0,1}\theta_{p+1,m-p-1}^\dagger)s-\theta_{p+1,m-p-1}^\dagger(\db_{E_{p,m-p} }s)=0,
\end{align*}
where the second  equality follows from $\theta_{p+1,m-p-1}^\dagger\wedge\theta_{m-p}^\dagger=0$, and the last one follows from $D_h^{0,1}(\theta^\dagger)=0$.  Here $D_h$ is the Chern connection for the Hodge bundle $(E=\oplus_{p+q=m}E_{p,q},h)$.  
Since $(F^{p+1}, h_{p+1}(\bm{\beta}))$ is also acceptable by \cref{lem:acceptable}, we can invoke \cref{thm:L2} to find some  $\sigma\in \sC^\infty(U, F^{p+1})$ such that 
$$
	\db_{F^{p+1}}(\sigma)=-\theta_{p+1,m-p-1}^\dagger s
$$
and \begin{align*}
	\int_{U} |\sigma|^2_{h_{p+1}(\bm{\beta},N)} d\mbox{vol}_{\omega_P}<\infty 
\end{align*}
for some $N\gg 1$.  Here $h_{p+1}(\bm{\beta},N)$ is a new metric for $F^{p+1}$ define by
$$
h_{p+1}(\bm{\beta},N)=h_{p+1}\cdot \prod_{i=1}^{\ell}|z_i|^{\beta_i}\cdot e^{-N\varphi}
$$
with $\varphi$ defined in \eqref{eq:potent}.  Since $|s|_{h_{p,m-p}}\lesssim \prod_{i=1}^{\ell}|z_i|^{-\alpha_i-\ep}$ for any $\ep>0$, it follows that  
 $  |s|_{h_{p,m-p}(\bm{\beta},N)}<C$ 
for some constant $C>0$,  where we define $$
h_{p,m-p}(\bm{\beta},N)=h_{p,m-p}\cdot \prod_{i=1}^{\ell}|z_i|^{\beta_i}\cdot e^{-N\varphi}.
$$
Thus  the section $\tilde{s}:=\Phi^{-1}([\sigma,s])$ is a holomorphic section of $F^p$ such that
 \begin{align*}
	\int_{U} |\tilde{s}|^2_{h_{p}(\bm{\beta},N)} d\mbox{vol}_{\omega_P}=\int_{U} |\sigma|^2_{h_{p+1}(\bm{\beta},N)} d\mbox{vol}_{\omega_P}+\int_{U} |s|^2_{h_{p,m-p}(\bm{\beta},N)} d\mbox{vol}_{\omega_P}<\infty 
\end{align*}
	  Since $(F^p, h_p(\bm{\beta}))$ is also acceptable by \cref{lem:acceptable}, thanks to \cref{lem:estimate},  
	 over some $U(r) $ for $0<r<1$  we have
$|\tilde{s}|_{h_p(\bm{\beta})}\lesssim   \prod_{j=1}^{\ell}|z_j|^{-\ep}$ for any $\ep>0$. Therefore, 
 $|\tilde{s}|_{h_p}\lesssim   \prod_{j=1}^{\ell}|z_j|^{-\beta_j-\ep}$ for any $\ep>0$. It follows that 
 $
 \tilde{s}\in \cP_{\bm{\beta}} F^p(X(r)).
 $ 
 By \eqref{continuous} we conclude that  $\tilde{s}\in \cP_{\bm{\alpha}} F^p(X(r'))$ for some   $0<r'<1$.    This implies the right exactness of \eqref{eq:exact} as  $q(\tilde{s})=s$. The theorem is proved. 
 \end{proof} 
Let us prove \cref{main}. 
\begin{proof}[Proof of \cref{main}]
Thanks to  \cref{lem:equal}, we have $V^{\rm Del}_{\bm{\alpha}}=\pa V$.  
    By \cref{lem:acceptable}, $(F^p, h_p)$ and $(E_{p,m-p}, h_{p,m-p})$ are acceptable bundles for any $p$. It follows from \cref{thm:Moc2} that the induced filtered bundle $\cP_{*}F^p$ and $\cP_{*} E_{p,m-p} $  defined in \cref{sec:pro} are parabolic ones. In particular, $\pa F^p$ and $\pa E_{p,m-p} $ are locally free sheaves.  Denote by  $j:X\backslash D\to X$  the inclusion map.  Note that  \begin{align}\label{eq:diff}
	\pa F^p=j_*(F^p)\cap \pa V\stackrel{\cref{lem:equal}}{=} j_*F^p\cap V_{\bm{\alpha}}^{\rm Del}=:F^p_{\bm{\alpha}}. 
\end{align}
Hence the exact sequence \eqref{eq:exact} in \cref{thm:exact} implies the following one
$$
0\to F^{p+1}_{\bm{\alpha}}\to F^p_{\bm{\alpha}}\to \pa E_{p,m-p}\to 0.
$$
In particular,  $F^p_{\bm{\alpha}}/F^{p+1}_{\bm{\alpha}}$ is isomorphic to $\pa E_{p,m-p}$, which is thus locally free.  The theorem is proved. 
\end{proof}
\subsection{On the nilpotent orbit theorem}
 In this subsection we apply \cref{main} to prove \cref{corx} following closely Schmid's original method \cite[p. 288-289]{Sch73}. 	  We will use the notations and conventions in \cref{sec:period}. 
 
 Let $(V,\nabla,F^\bullet,Q)$ be a complex polarized variation of Hodge structure on $(\Delta^*)^p\times \Delta^q$. Denote by $\Phi:\bH^{p}\times \Delta^q\to \sD$ its period mapping, where we set
 \begin{align*}
 \bH^{p}\times \Delta^q&\to\Delta^n\\
 (z,w)&\mapsto (e^{z_1},\ldots,e^{z_p},w)
 \end{align*} 
to be the uniformizing map. 
 Let $T_j$ be the 
 monodromy transformation   defined in \cref{sec:Del}. For some fixed $\bm{\alpha}\in \bR^p$,  there exist $S_i, N_i\in \End(V^\nabla)$ such that
 \begin{itemize}
 	\item $T_i=\exp (2\pi i (S_i+N_i))$;
 	\item $[S_i,S_j]=0$, $[S_i,N_j]=0$, and $[N_i, N_j]=0$;
 	\item $S_i$ is semisimple whose eigenvalues lying in $(\alpha_i-1,\alpha_i]$ and $N_i$ is nilpotent. 
 \end{itemize}
Let us define $$
\tilde{\Psi}(z,w):=\exp(\sum_{i=1}^{p}(S_i+N_i)z_i)\Phi(z,w),
$$
which satisfies $\tilde{\Psi}(z_1,\ldots,z_i+2\pi i,\ldots, z_p,w)=\tilde{\Psi}(z,w)$ for $i=1,\ldots,p$.   It  thus descends to a single valued  map  
   $\Psi:(\Delta^*)^p\times \Delta^q\to \check{\sD}$ such that $\Psi(e^{z_1},\ldots,e^{z_p},w)=\tilde{\Psi}(z,w)$. 
   \begin{lem}
   The twisted holomorphic map	 $\Psi $  extends holomorphically to $\Delta^{p+q}$.
   \end{lem} 
\begin{proof}
We use the notations in \cref{sec:Del}.  We fix a basis $v_1,\ldots,v_r$ of $V^\nabla$ such that each $v_i$ belongs to some $\bE_{\bm{\lambda}}$. Then the  sections $\tilde{v}_1,\ldots,\tilde{v}_r$ defined in \eqref{eq:twist} induces a trivialization 
$$
V^\nabla\otimes_\bC  \Delta^{p+q}\simeq V^{\rm Del}_{\bm{\alpha}},
$$ 
where $V^{\rm Del}_{\bm{\alpha}}$ is the Deligne extension. Under such trivialization, the Hodge filtration $F^\bullet_{\bm{\alpha}}(t_1,\ldots,t_{p+1})$   becomes    $\Psi(t_1,\ldots,t_{p+1})$.   Thanks to \cref{main},  the Hodge filtration $F^\bullet_{\bm{\alpha}}$ extends to locally free sheaves over $\Delta^{p+q}$ such that $F^p_{\bm{\alpha}}/F^{p+1}_{\bm{\alpha}}$ is also  locally free. Therefore, $\Psi$ extends holomorphic maps over $\Delta^{p+q}$. 
\end{proof}
This lemma thus proves  \cref{extension}. We write $a(w):=\Psi(0,w)$.  In general it does not lie in $\sD$.  
 
 The following well-known result follows from the fact that $GL(V^\nabla)$ acts transitively on $\check{\sD}$. 
 \begin{lem}\label{lem:left}
For any $g\in GL(V^\nabla)$, consider the left translation $L_g:\check{\sD}\to \check{\sD}$ with $L_g(F):=g  F$.  Then 
\begin{align*}
	(L_g)_*: T_{\check{\sD},F}^{-1,1}\stackrel{\sim}{\to} T_{\check{\sD},g  F}^{-1,1}.
\end{align*} \qed
 \end{lem}

 Recall that for any $A\in \End(V^\nabla)$ and any $F\in \check{\sD}$, we denote by $[A]_F$ the image of $A$ under the natural map $\End(V^\nabla)\to T_{\check{\sD},F}$. 
\begin{lem}\label{lem:hor}
For each  $i=1,\ldots,p$, 	  $[S_i+N_i]_{a(w)}\subset T_{\check{\sD},a(w)}^{-1,1}$.   
\end{lem} 
\begin{proof}
Since
$$
\tilde{\Psi}_*(\frac{\d}{\d z_i})(z,w)=[S_i+N_i]_{\tilde{\Psi}(z,w)}+ (L_{\exp(\sum_{i=1}^{p}(S_i+N_i)z_i)})_*\Phi_*(\frac{\d}{\d z_i})(z,w)
$$
$\Phi_*(\frac{\d}{\d z_i})$ is horizontal since $\Phi$ is a horizontal mapping by \cref{sec:period}.  By \cref{lem:left}  $(L_{\exp(\sum_{i=1}^{p}(S_i+N_i)z_i)})_*\Phi_*(\frac{\d}{\d z_i})(z,w)$ is horizontal. 
On the other hand,
$$
\tilde{\Psi}_*(\frac{\d}{\d z_i})(z,w)=\Psi_*(\frac{\d}{\d t_i})(e^{z_1},\ldots,e^{z_p},w)\cdot e^{z_i}
$$
which tends to zero if $\Re z_i\to -\infty$ and $\Re z_j\leq C$ for other  $j$. By the  continuity, this implies that
$$
 [S_i+N_i]_{a(w)}\in T_{\check{\sD},a(w)}^{-1,1}.  
$$
\end{proof}
We are ready to prove \cref{horizontal}. 
 \begin{lem}
The holomorphic mapping \begin{align*}
\vartheta: \bH^{p}\times \Delta^q&\to\check{\sD}\\
(z,w)&\mapsto \exp(-\sum_{i=1}^{p}z_i(S_i+N_i))\circ a(w)
\end{align*}  is horizontal.
 \end{lem}
\begin{proof}
Note that	$\vartheta_*(\frac{\d}{\d z_i})=[S_i+N_i]_{\vartheta(z,w)}$.   Since $[S_i,N_i]=0$, one has
$$
 (L_{\exp(\sum_{i=1}^{p}(S_i+N_i)z_i)})_*\big([S_i+N_i]_{\vartheta(z,w)}\big)=\Big[\Ad_{\exp(\sum_{i=1}^{p}(S_i+N_i)z_i)}(S_i+N_i)\Big]_{a(w)}=[S_i+N_i]_{a(w)}. 
$$
It then follows from \cref{lem:left,lem:hor} that $ [S_i+N_i]_{\vartheta(z,w)}\in T_{\check{\sD},\vartheta(z,w)}^{-1,1}$.  We conclude that $\vartheta_*(\frac{\d}{\d z_i})$ is horizontal. 

On the other hand,  one has 
$$\vartheta_*(\frac{\d}{\d w_i})=(L_{\exp(-\sum_{i=1}^{p}(S_i+N_i)z_i)})_*a_*(\frac{\d}{\d w_i}),$$  and $$
 {\Psi}_*(\frac{\d}{\d w_i})(e^z,w)=\tilde{\Psi}_*(\frac{\d}{\d w_i})(z,w):=(L_{\exp(\sum_{i=1}^{p}(S_i+N_i)z_i)})_*\Phi_*(\frac{\d}{\d w_i})(z,w).
$$
  Since $\Phi_*(\frac{\d}{\d w_i})$ is horizontal,  by \cref{lem:left}, $\tilde{\Psi}_*(\frac{\d}{\d w_i})(z,w):=(L_{\exp(\sum_{i=1}^{p}(S_i+N_i)z_i)})_*\Phi_*(\frac{\d}{\d w_i})(z,w)$ is horizontal, and thus  ${\Psi}_*(\frac{\d}{\d w_i})(e^z,w)$ is horizontal.  Letting $\Re z_i\to -\infty$ for $i=1,\ldots, p$,   we conclude that
$$
 {\Psi}_*(\frac{\d}{\d w_i})(0,w)=a_*(\frac{\d}{\d w_i})
$$ is also horizontal. 
We apply \cref{lem:left} again to conclude that
 $\vartheta_*(\frac{\d}{\d w_i})$ is horizontal. In conclusion, $\vartheta$ is a horizontal mapping. We proved \cref{horizontal}.
\end{proof}

  The rest of the paper is devoted to the proof of  \cref{one var}.  We will only  deal with the case of one variable. We first start a lemma  in linear algebra whose proof is direct (cf. \cite[\S 7.5]{SS22} for a detailed proof).   
 \begin{lem}\label{lem:li}
 Let $S\in \End(V^\nabla)$ be semisimple with real eigenvalues. Then there exists a constant $C>0$ such that
 $$
 \lVert \Ad e^{xS}\rVert\leq C\exp(( \lambda_{\max}-\lambda _{\min})\cdot |x|) \quad \mbox{for all}\quad x\in \bR, 
 $$
 where $\lambda_{\max}$ and $\lambda_{\min}$ are the largest and smallest eigenvalue of $S$. 
 Let $N\in \End(V^\nabla)$ be nilpotent. Then
 $$
 \lVert \Ad e^{xN}\rVert\leq C'|x|^m
 $$
 for some $C',m>0$.
 
 Here we fix a reference polarized Hodge structure $o\in \sD$ which induces  metrics for $V^\nabla$ and $\End(V^\nabla)$.   $\lVert \Ad e^{xS}\rVert$ is the operator norm with respect  to such metric of $\End(V^\nabla)$. \qed
 \end{lem} 

The following two lemmas are due to Schmid \cite[Lemmas 8.12 \& 8.19]{Sch73}.  They are stated for period domains of real Hodge structures. However, their proofs can be generalized to period domains of complex polarized Hodge structures verbatim, and we thus omit their proofs here. 
\begin{lem}\label{s1}
	If $g\in GL(V^\nabla)$, then for some   natural  distance $d_{\check{\sD}}$  of $\check{\sD}$,  we have
	$$
	d_{\check{\sD}}(ga,gb)\leq \lVert \Ad g \rVert d_{\check{\sD}}(a,b)
	$$
	for any points $a,b\in \check{\sD}$.  \qed
\end{lem}
\begin{lem}\label{lem:Schmid}
 Let  $\Phi:\bH\to \sD$	 be the period map associated to a complex polarized variation of Hodge structure on $\Delta^*$.  Fix $\alpha,k>0$ and a reference point $o\in \sD$. Choose $g(z)\in G=U(V^\nabla,Q)$ such that $g(z)\cdot o=\Phi(z)$. Then there exist $C,\beta>0$ such that   if $|\Im z|\leq k$,  one has 
	$$
	\lVert \Ad g(z)^{-1}\rVert\leq  C|\Re z|^\beta
	$$
for $\Re z<-\alpha$.  Here	$\lVert \Ad g(z)^{-1}\rVert$ is the operator norm defined in \cref{lem:li}. \qed
\end{lem}
By \cite[\S 4.1]{SS22}, we know that  $G:=U(V^\nabla,Q)$ acts transitively on the period domain $\sD$, and   $\sD$ admits a natural $G$-invariant  distance $d_{\sD}$. 
 
\begin{proof}[Proof of \cref{one var}]
Let $T\in GL(V^\nabla)$ be the monodromy operator associated to the counter-clockwise generator of  $\pi_1(\Delta^*)$.   Note that $T\in G:=U(V^\nabla,Q)$. Recall that there exist commuting $S,N\in GL(V^\nabla)$ such that 
\begin{itemize}
	\item $\exp(2\pi i(S+N))=T$;
	\item $S$ is semisimple with eigenvalues lying in $(\alpha-1,\alpha]$;
	\item $N$ is nilpotent.  
\end{itemize} Denote by $a=\Psi(0)$.  Then  for $|t|$ small enough, one has
$$d_{\check{\sD}}(a, \Psi(t)  )< C|t|\quad \mbox{for some}\ \ C>0,$$
which is equivalent to that
\begin{align}\label{eq:dis}
	d_{\check{\sD}}(a, \Psi(e^z)  )< C e^x 
\end{align}
when $x\leq -M$ for some $M>0$. Here we write $z=x+i y$. 
Assume now $|y|\leq 2\pi$ and $x\leq -M$.   Then 
\begin{align*}
d_{\check{\sD}}(\exp(-(S+N)z)a, \Phi(z))&\leq 
\lVert\Ad \exp((S+N)z)\rVert \cdot  d_{\check{\sD}}(a, \Psi(e^z)  ) \\
&\leq  \lVert\Ad \exp(Nx)\rVert\cdot  \lVert\Ad \exp(i(S+N) y)\rVert\cdot \lVert\Ad \exp(S x)\rVert \cdot  d_{\check{\sD}}(a, \Psi(e^z)  )  \\
&\leq C_1 \lVert\Ad \exp(Nx)\rVert \cdot  \lVert\Ad \exp(S x)\rVert  \cdot  d_{\check{\sD}}(a, \Psi(e^z)  ) \\
&\leq C_2 |x|^{m}\cdot  \exp( (\lambda_{\max}-\lambda_{\min})\cdot|x|)\cdot d_{\check{\sD}}(a, \Psi(e^z)  )\\
&\leq C_3    |x|^{m}\cdot  \exp( (\lambda_{\max}-\lambda_{\min})\cdot|x|)\cdot  e^x 
 \leq C_3 |x|^me^{\delta x}.
\end{align*}
The first inequality is due to \cref{s1}, the third one holds since $|y|\leq 2\pi$, the fourth one follows from \cref{lem:li}, and the fifth one follows from \eqref{eq:dis}.   Here $\lambda_{\max}$ and $\lambda_{\min}$ are the largest and smallest eigenvalue of $S$. Therefore,  $\lambda_{\max}-\lambda_{\min}<1$ and thus the last inequality can be achieved for some $\delta>0$. Here $C_1,\ldots,C_3>0$ are some positive constants.  

Fix a reference point $o\in \sD$ and let $g(z)\in G$ such that $g(z)\cdot o=\Phi(z)$. By \cref{s1,lem:Schmid}  one gets
\begin{align}\label{eq:final}
	d_{\check{\sD}}(g(z)^{-1}\exp(-(S+N)z)a, o)&\leq  \lVert \Ad g(z)^{-1} \rVert\cdot d_{\check{\sD}}(\exp(-(S+N)z)a, \Phi(z))\\\nonumber
	& \leq C_4 |x|^{m+\beta}e^{\delta x}
\end{align}
if $|y|\leq 2\pi$ and $x<-M_2$ for some $M_2\geq M$ and $C_4,\beta>0$.  
Pick a small neighborhood $U$ of $o$ in $\sD$ such that  the distance  functions $d_{\sD}$ and $d_{\check{\sD}}$ are mutually bounded over $U$. By \eqref{eq:final} when $|y|\leq 2\pi$, $x\leq -M_3$ for some $M_3\geq M_2$,  $g(z)^{-1}\exp(-(S+N)z)a$ will be entirely contained in $U$.   Note that $g(z)\in G$, it follows that  $ \exp(-(S+N)z)a\in \sD$  if $|y|\leq 2\pi$ and $x\leq -M_3$. When $|y|>2\pi$ and $x\leq -M_3$, we find some integer $\ell$ such that $|y-2\pi \ell |\leq 2\pi$. Then  $\exp(-(S+N)(z-2\pi i\ell))a\in \sD$. Since $ \exp(-(S+N)z)a=T^{-\ell}\exp(-(S+N)(z-2\pi i\ell))a$ and $T\in G$, it follows that $ \exp(-(S+N)z)a\in \sD$. In conclusion, 
$ \exp(-(S+N)z)a\in \sD$  if  $x\leq -M_3$. We prove the first claim in  \cref{one var}.

 Recall that the distance  functions $d_{\sD}$ and $d_{\check{\sD}}$ are mutually bounded over $U$.  By \eqref{eq:final} again for some $C_5>0$  we have
$$
d_{ {\sD}}(g(z)^{-1}\exp(-(S+N)z)a, o)\leq C_5 |x|^{m+\beta}e^{\delta x}.
$$
for $|y|\leq 2\pi$, $x\leq -M_3$. Since the action of $g(z)$ is $d_{\sD}$-distance invariant, we obtain the distance estimate
$$
d_{ {\sD}}( \exp(-(S+N)z)a, \Phi(z))\leq C_5 |x|^{m+\beta}e^{\delta x}
$$
for $|y|\leq 2\pi$, $x\leq -M_3$.   When $|y|>2\pi$ and $x\leq -M_3$, one picks some integer $\ell$ such that $|y-2\pi \ell |\leq 2\pi$. Then 
$$
d_{ {\sD}}( \exp(-(S+N)(z-2\pi i\ell))a, \Phi(z-2\pi i\ell))\leq C_5 |x|^{m+\beta}e^{\delta x}.
$$
 In other words, 
 $$
 d_{ {\sD}}( T^\ell\exp(-(S+N)z)a, T^\ell\Phi(z))\leq C_5 |x|^{m+\beta}e^{\delta x}.
 $$
 As $T$ is also $d_{\sD}$-distance invariant,  it follows that 
 $$
 d_{ {\sD}}(  \exp(-(S+N)z)a,  \Phi(z))\leq C_5 |x|^{m+\beta}e^{\delta x}.
 $$
 for $x\leq -M_3$. The distance estimate is obtained. 
\end{proof}

\medspace

\noindent{\bf Acknowledgment.} This paper is motivated by a question asked by  Junchao Shentu, to whom I express my gratitude.  I especially thank Christian Schnell for his personal notes and discussions on the nilpotent orbit theorem in 2020 and 2022, and sending me the preprint \cite{SS22}.  I also thank the referees  for their valuable suggestions, which have greatly improved the   paper.  This work is partially supported by the French Agence Nationale de la Recherche (ANR) under reference ANR-21-CE40-0010.

\end{document}